\documentclass[11pt]{article}
\usepackage{amssymb,amsmath,latexsym}
\usepackage{amsbsy}
\usepackage{amsfonts}
\usepackage{amsopn}
\usepackage{amsthm}
\usepackage{amsxtra}
\usepackage{amsmath,verbatim}
\usepackage{subfigure}
\usepackage[colorlinks=true, pdfstartview=FitV, linkcolor=black, citecolor=black, urlcolor=black]{hyperref}
\usepackage[dvips]{graphicx}
\usepackage{nicefrac}
\usepackage{amscd}
\usepackage{dsfont}

\newcommand{\R}{{\mathbb R}}
\newcommand{\rp}{\R^+}
\newcommand{\rpn}{\R^+_0}

\newcommand{\N}{{\mathbb N}}

\newcommand{\dd}{\text{d}}
\newcommand{\rcll}{\text{\rm RCLL}}

\theoremstyle{plain}

\newtheorem{theorem}{Theorem}

\newtheorem{proposition}[theorem]{Proposition}
\newtheorem{lemma}[theorem]{Lemma}
\newtheorem{definition}[theorem]{Definition}
\newtheorem{example}[theorem]{Example}

\newtheorem{remark}[theorem]{Remark}
\newenvironment{rem}{\begin{remark}}{\hfill$\lozenge$\end{remark}}
\newenvironment{pf}{\begin{proof}}{\hfill$\square$\end{proof}}
\newtheorem{hypothesis}{Hypothesis}

\allowdisplaybreaks

\begin{document}

\renewcommand{\qedsymbol}{$\blacksquare$}

\newcommand{\D}{{\rm d}}
\def\ee{\varepsilon}
\def\qed{{\hfill $\Box$ \bigskip}}
\def\MM{{\cal M}}
\def\BB{{\cal B}}
\def\LL{{\cal L}}
\def\FF{{\cal F}}
\def\EE{{\cal E}}
\def\QQ{{\cal Q}}

\def\R{{\mathbb R}}
\def\L{{\bf L}}
\def\E{{\mathbb E}}
\def\F{{\bf F}}
\def\P{{\mathbb P}}
\def\N{{\mathbb N}}
\def\eps{\varepsilon}
\def\wh{\widehat}
\def\pf{\noindent{\bf Proof.} }

\title{\Large \bf Survival of homogeneous fragmentation processes with killing}
\author{Robert Knobloch\thanks{Institut f\"ur Mathematik, Goethe-Universit\"at Frankfurt am Main, 60054 Frankfurt am Main, Germany
\newline
e-mail: knobloch@math.uni-frankfurt.de} \,and  Andreas E. Kyprianou\thanks{Department of Mathematical
Sciences, University of Bath, Claverton Down, Bath, BA2 7AY, U.K.
\newline
e-mail: a.kyprianou@bath.ac.uk}  }
\date{\today}

\maketitle

\begin{abstract}
We consider a homogeneous fragmentation process with killing at an exponential barrier. With the help of two families of martingales we analyse the decay of the largest fragment for parameter values that allow for survival. In this respect the present paper is also concerned with the probability of extinction of the killed process.
\end{abstract}



\section{Introduction and main results}

This paper is concerned with a homogeneous fragmentation process in which there is an additional killing upon crossing a certain space-time barrier (this killing mechanism is defined rigorously in Section~\ref{s.killed_fragmentations}). In particular, we consider the decay of the largest fragment in this process with killing. 

The motivation for the killing procedure that we introduce in the present paper, partly stems from its relation to the Fisher-Kolmogorov-Petrovskii-Piskounov (FKPP) equation. In the context of fragmentation processes this connection is studied in \cite{Kno12}.
The role an analogous killing plays with regard to solutions of the FKPP equation in the setting of branching Brownian motions (BBM) was investigated in \cite{harrisetal}. Furthermore, this kind of killing for random multi-particle systems was considered also in various other contexts and in the literature there is some interesting recent activity in this regard.  The killing of BBM at  a linear space-time barrier  was also studied in \cite{HH07} and recently in \cite{BBS11}, where a relation of the killed BBM and its genealogy to continuous-state branching processes and the Bolthausen-Sznitman
coalescent was revealed. Regarding  similar killing schemes for branching random walks we refer e.g. to  \cite{BG10}, \cite{DS07}, \cite{DS08} and \cite{GHS11}. In the context of fragmentation processes such a killing mechanism has not been considered so far. However, the above-mentioned papers which are concerned with related types of spatial branching processes suggest that  this kind of killing  has interesting applications.

We begin our exposition by briefly reviewing what is meant by a homogeneous fragmentation process, thereby  introducing some notation. 

\subsection{homogeneous fragmentation processes}

Below we give a brief overview of the definition and structure of a homogeneous fragmentation process. The reader is referred to Bertoin \cite{bertfragbook} for a more detailed overview.
Let $\mathcal{P}$ be the space of partitions of the natural numbers. Here a partition of $\mathbb{N}$ is a sequence $\pi=(\pi_1, \pi_2, \cdots)$ of disjoint sets, called blocks, such that $\bigcup_{i\in\N} \pi_i = \mathbb{N}$. The blocks of a partition are enumerated in the increasing order of their least element, that is to say $\min \pi_i\leq \min\pi_j$ when $i\leq j$ (with the convention that $\min \emptyset = \infty$). 
Now consider the measure $\mu$ on $\mathcal{P}$, given by
\[
\mu(d\pi) = \int_{\mathcal{S}}\varrho_{\bf s}(d\pi)\nu(d{\bf s}),
\]
where $\varrho_{\bf s}$  is the law of Kingman's paint-box based on ${\bf s}\in\mathcal{S}$ (cf. page~98 of Bertoin \cite{bertfragbook}) with
 \[
\mathcal{S}: = \left\{\mathbf{s}=(s_1, s_2, \cdots) : s_1\geq  s_2 \geq \cdots \geq 0 , \, \sum_{i=1}^\infty s_i \le 1\right\},
 \] 
and the so-called {\it dislocation measure} $\nu\neq 0$  is a measure on $\mathcal{S}$
 such that 
\begin{equation} 
\nu({\bf s}\in\mathcal{S}:s_2=0)=0
\label{e.levymeasure.0}
\end{equation}
as well as
\begin{equation}
\int_{\mathcal{S}}(1-s_1)\nu(d{\bf s})<\infty.
\label{e.levymeasure}
\end{equation}
 It is known that $\mu$ is  an exchangeable partition measure, meaning that it is invariant under the action of finite permutations on $\mathcal{P}$. It is also known (cf. Chapter 3 of Bertoin \cite{bertfragbook}) that it is possible to construct a fragmentation process  on the space of partitions $\mathcal{P}$ with the help of  a Poisson point process $\{(\pi(t), k(t)) : t\geq 0\}$ on $\mathcal{P}\times \mathbb{N}$ which has intensity measure $\mu\otimes\sharp$, where $\sharp$ is the counting measure. The aforementioned   $\mathcal{P}$-valued fragmentation process is a  Markov process which we denote by  $\Pi = \{\Pi(t) : t\geq 0\}$, where
 $\Pi(t) = (\Pi_1(t), \Pi_2(t),\cdots)\in\mathcal{P}$ 
 is such that at all times $t\geq0$ for which an atom $(\pi(t), k(t))$ occurs in $(\mathcal{P}\backslash(\mathbb{N},\emptyset,\ldots))\times \mathbb{N}$, $\Pi(t)$ is obtained from $\Pi(t-)$ by partitioning the $k(t)$-th block  into the sub-blocks $(\Pi_{k(t)}(t-) \cap \pi_j(t) : j=1,2,\cdots)$. When $\nu$ is a finite measure each block experiences an exponential holding time before it fragments. 

Thanks to the properties of the exchangeable partition measure $\mu$ it can be shown that for each $t\geq 0$ the distribution of $\Pi(t)$ is exchangeable and that the  blocks of $\Pi(t)$  have asymptotic frequencies in the sense that for each $i\in\mathbb{N}$ the limit 
\[
|\Pi_i(t)|:=\lim_{n\to\infty} \frac{1}{n}\sharp\{\Pi_i(t)\cap \{1,\cdots, n\} \}
\]
exists almost surely.  Moreover, Bertoin showed that $|\Pi_i(t)|$ exists $\mathbb P$--a.s. simultaneously for all $t\ge0$ and $i\in\mathbb N$. 

We denote the countable random jump times of $\Pi$ by 
$\mathcal I\subseteq\rpn$. Further, let $\mathcal F:=(\mathcal F_t)_{t\in\rpn}$ denote the filtration generated by $\Pi$. In addition, let $\mathcal G:=(\mathcal G_t)_{t\in\rpn}$ be the sub--filtration generated by the asymptotic frequencies of $\Pi$ and let $\mathcal F^1:=(\mathcal F^1_t)_{t\in\rpn}$ denote the filtration generated by $(\Pi_1(t))_{t\in\rpn}$. 

Let us define $\xi(t)  := -\log |\Pi_1(t)|$ for every $t\geq 0$, with the convention $-\log0:=\infty$. Resorting to the Poissonian construction of the fragmentation process, Bertoin proved that $\xi = \{\xi(t) : t\geq 0\}$ is a killed
subordinator with cemetery state $\infty$ and killing rate
\begin{equation}\label{e.kappa}
\kappa:=\int_{\mathcal S}\left(1-\sum_{k\in\N}s_k\right)\nu(\dd{\bf s}).
\end{equation}
Moreover, it is well known that its Laplace exponent $\Phi$, given by 
\[
e^{-\Phi(p)} : = \mathbb{E}(e^{-p\xi(1)}),
\] 
can be characterised over an appropriate domain of $p$ through the dislocation measure $\nu$ as follows. Define the constant 
\[
\underline{ p} = \inf\left\{  p\in \mathbb{R} : \int_{\mathcal{S}} \left| 1- \sum_{i=1}^\infty s_i^ {1+p} \right| \nu(d{\bf s}) <\infty\right\}
\]
which is necessarily in $[-1,0]$.
 Then
\[
\Phi( p) = \int_{\mathcal{S}} \left(  1- \sum_{i=1}^\infty s_i^ {1+p} \right)\nu(d{\bf s})
\]
for all $p>\underline{p}$ (and we understand $\Phi(\underline{p}) = \Phi(\underline{p}+)$). The tagged fragment $\Pi_1$, and in particular its Laplace exponent $\Phi$, can be used to extract information about the decay and spatial distribution of blocks in the fragmentation process. A case in point concerns the asymptotic rate of decay  of the largest block 
\[
\lambda_1(t) : = \sup_{n\in\N}|\Pi_n(t)|, \, t\geq 0.
\]
To this end, note that $\Phi$ is strictly increasing, concave and differentiable. We shall assume that  
\begin{equation}
\label{e.bar_p}
(p+1)\Phi'(p)>\Phi(p) \text{ for some }p\in(\underline{p}, \infty).
\end{equation}
This assumption is automatically satisfied if there exists some $p^*\ge\underline p$ with $\Phi(p^*)=0$, hence in particular in the conservative case where $\nu(\mathbf{s}\in\mathcal{S}: \sum_{k\in\mathbb{N}}s_k < 1)=0$ and thus $p^*=0$.
Following the reasoning in the proof of Lemma~1 in \cite{Be4} one may proceed with (\ref{e.bar_p}) in hand to show that  there exists a unique maximal value of the function 
\[
p\mapsto c_p := \frac{\Phi(p)}{p+1}
\]
in $(\underline{p},\infty)$, which is achieved at some  $\bar p>\underline{p}$ and which is also equal  to $\Phi'(\bar p)$.
%
This maximal value  turns out to characterise the asymptotic rate of decay of the largest block, as shown in the  following proposition that is lifted from Bertoin \cite{bertfragbook}.
\begin{proposition}[cf. Corollary~1.4 of \cite{bertfragbook}]\label{largestfragmentnotkilled}
We have 
\[
\lim_{t\to\infty}\frac{-\log\lambda_1(t)}{t}=c_{\bar p}
\]
$\mathbb P$--almost surely.
\end{proposition}
In Corollary~1.4 of \cite{bertfragbook} Bertoin proves this result for fragmentation chains, but in view of Lemma~1.35 of \cite{Kno11} the same line of argument works for fragmentation processes.

\subsection{Killed homogeneous fragmentation processes}\label{s.killed_fragmentations}

Now let $c>0$ and $x\in\rpn$. We want to introduce killing of $\Pi$ upon hitting the space--time barrier 
\[
\left\{(y,t)\in\rpn\times\rpn:y<e^{-(x+ct)}\right\}
\] 
as follows. A block $\Pi_n(t)$ is killed  at the moment of its creation $t\in\mathcal{I}$ if $|\Pi_n(t)|<e^{-(x+ct)}$, see Figure~\ref{f.kfp.1}. Here, killing a block means that it is sent to a {\it cemetery state}, which we shall identify by  $\emptyset$. 
\begin{figure}[htb]
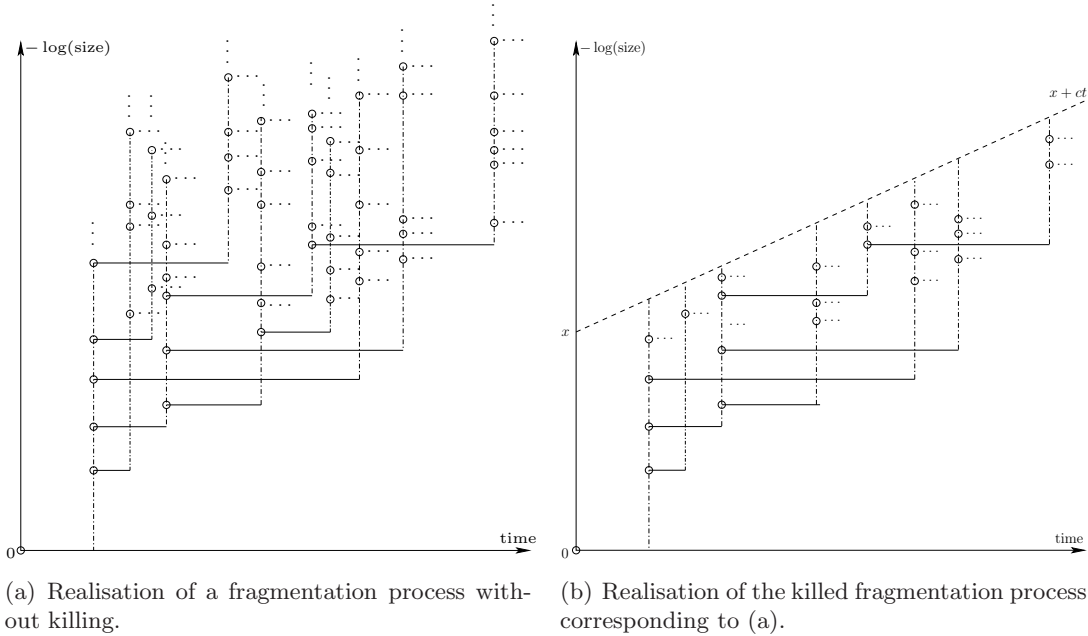
\centering
\subfigure[Realisation of a fragmentation process without killing.]{\resizebox{7cm}{!}{\input{pic3long_125_log.pstex_t}}}\quad
\subfigure[Realisation of the killed fragmentation process corresponding to (a).]{\resizebox{7cm}{!}{\input{pic2b_125_log.pstex_t}}}\quad
\caption[Killed fragmentation process]{Realisation of a  fragmentation process with finite dislocation measure without killing, in (a), and with killing, in (b).}\label{f.kfp.1}
\end{figure}

Suppose that for $t\geq 0$ we define  $\mathcal N^x_t$ to be the index set of the blocks in $(\Pi_n(t))_{n\in\N}$ that are not yet killed by time $t$. It is important to note that $N^x_t: = \text{card}(\mathcal N^x_t)$ is finite for each $t$. Indeed, as $\sum_{n\in\N}|\Pi_n(t)|\le1$ we infer that $|\Pi_n(t)|\ge e^{-(x+ct)}$ for at most $e^{x+ct}$--many $n\in\N$. That is 
\[
N^x_t\le e^{x+ct} 
\]
for all $t\in\rpn$.  Denote  by $\Pi^x:=(\Pi^x(t): t\geq 0)$, where $\Pi^x(t) =
(\Pi_n(t))_{n\in\mathcal{N}^x_t}$, the resulting killed fragmentation process and note that $\Pi^x$ is not $\mathcal{P}$-valued. 

For each $n\in\N$ the block of $\Pi^x$ containing  $n
$ has a killing time that may be finite or infinite. 
Note that the killed fragmentation process $\Pi^x$ also depends on the constant $c>0$. However, in order to keep the notation as simple as possible we do not include the parameter $c$ in the notation as this constant does not change within the results or proofs of this paper.

In this paper we shall answer the question whether it is possible that the supremum   over all the aforementioned respective individual killing times, which is henceforth denoted by $\zeta^x$, is finite.
We say that $\Pi^x$ becomes {\it extinct} if $\{\zeta^x<\infty\}$. 
Note that $\zeta^x>0$ $\mathbb P$--a.s., that is to say almost surely instantaneous extinction is not possible, on account of the fact that the spectrally negative L\'evy process $\{ct +\log |\Pi_1(t)|: t\geq 0\}$ is irregular for $(-\infty,0)$ when issued from the origin.
Our first main result in this respect is the following.

\begin{theorem}\label{p.ext-probab}
For all $c\le c_{\bar p}$ we have $\mathbb P(\zeta^x<\infty)=1$ for every $x\in\rpn$. If $c>c_{\bar p}$, then $x\mapsto\mathbb P(\zeta^x<\infty)$ is a 
  nonincreasing, $(0,1)$--valued function on $\rpn$.
\end{theorem}

\noindent In the case that extinction does not occur with probability 1, we shall give two qualitative results concerning the evolution of the process on survival. The first result shows that the total number of fragments in the surviving process explodes.
\begin{theorem}\label{l.p.2.2}
Let $c>c_{\bar p}$. Then we have that
\[
\limsup_{t\to\infty}N^x_t=\infty
\]
holds  $\mathbb P(\cdot|\zeta^x=\infty)$--a.s. for any $x\in\rpn$.
\end{theorem}

The second result shows that the asymptotic exponential rate of decay of the largest fragment,
\[
\lambda^x_1(t): = \max_{n\in\N}|\Pi_n^x(t)|, \, t\geq 0,
\] 
is the same  as when the killing scheme is not in effect, cf. Proposition~\ref{largestfragmentnotkilled}.

\begin{theorem}\label{l.fp.1}
Let $c>c_{\bar p}$ and $x\in\rpn$. Then we have
\[
\lim_{t\to\infty}\frac{-\log\lambda^x_1(t)}{t}=c_{\bar p}
\]
$\mathbb P(\cdot|\zeta^x=\infty)$--almost surely.
\end{theorem}

What lies fundamentally behind the proofs of our main results is a detailed study of the interaction between two classes of martingales. 

The outline of this paper is as follows. In the next section we provide some general notions that are used in the subsequent parts of the present paper and in particular we employ the connection between fragmentations and L\'evy processes. Section~\ref{s.properties} is concerned with the proof of Theorem~\ref{p.ext-probab} and in Section~\ref{proof of l.p.2.2} we provide the proof of Theorem~\ref{l.p.2.2} . Then, in Section~\ref{s.mm}, we introduce a multiplicative process and examine when this process is a martingale. The object under consideration in Section~\ref{s.am}  is an additive process which also turns out to be a martingale and whose limit we study with regard to strict positivity. In the final section of this paper we prove Theorem~\ref{l.fp.1}. 

\section{Preliminaries}\label{preliminaries}

Let $B_n(t)$, $t\in\rpn$, denote the block in $\Pi(t)$ that contains the element $n\in\N$ 
and 
recall from (\ref{e.kappa}) that under $\mathbb P$ the process $\xi_n=(-\log|B_n(t)|)_{t\in\rpn}$ is a killed
subordinator (with cemetery state $+\infty$ and killing rate $\kappa$).
\begin{definition}
For every $n\in\N$ let the process $X_n:=(X_n(t))_{t\in\rpn}$\label{p.Xn} be defined by 
\[
X_n(t):=ct-\xi_n(t)
\] 
for all $t\in\rpn$.
\end{definition} 

Notice that under $\mathbb P$ the dynamics of the process $X_n$ are those of a killed
spectrally negative L\'evy process of bounded variation (with cemetery state $-\infty$ and killing rate $\kappa$). 
Moreover, the jump times of $X_n$, henceforth denoted by the countably infinite set, 
$\mathcal{I}_n\subseteq\rp$, 
are the set of dislocation times of $(B_n(t))_{t\in\rpn}$. That is, $X_n$ jumps exactly when the subordinator $\xi_n$ jumps. 
For any $n\in\N$ and $x\in\rpn$ consider the following $\mathcal F$--stopping times:
\[
\tau^+_{n,x}:=\inf\{t\in\rpn:X_n(t)>x\}\qquad\text{as well as}\qquad
\tau^-_{n,x}:=\inf\{t\in\rpn:X_n(t)<-x\}.\label{p.tau^x^n_0}
\]

For any $p\in(\underline p,\infty)$ consider the change of measure given by
\begin{equation}
\left.\frac{\dd\mathbb P^{(p)}}{\dd\mathbb P}\right|_{\mathcal F_t}=e^{\Phi(p)t-p\xi(t)}=e^{pX_1(t)-\psi(p)t},
\label{1DCOM}
\end{equation}
where 
\[
\psi(p)=\frac{1}{t} \log \mathbb{E}(e^{p X_1(t)}) = cp-\Phi(p)  
\]
is the Laplace exponent of $X_1$.
Moreover, considering the projection of  (\ref{1DCOM}) onto the sub--filtration $\mathcal G$ results in
\[
\left.\frac{\dd\mathbb P^{(p)}}{\dd\mathbb P}\right|_{\mathcal G_t}
=M_t(p): = \sum_{n\in\mathbb{N}} |\Pi_n(t)|^{1+p}e^{\Phi(p)t},
\]
for all $p\in(\underline p,\infty)$ and $t\in\rpn$. 
\begin{rem}\rm\label{r.equivalentmeasures}
Let $p\in(\underline p,\bar p)$ and denote by $M_\infty(p)$ the $\mathbb P$--a.s. limit of the nonnegative martingale $M(p):=(M_t(p))_{t\in\rpn}$. According to Theorem~1 of \cite{BR03} (cf. also Theorem~4 of \cite{BR05} for the conservative case) the unit--mean martingale $M(p)$ is uniformly integrable. Hence, $\mathbb E(M_\infty(p))=1$ and thus $\mathbb P^{(p)}$ is a probability measure on $\mathcal G_\infty:=\bigcup_{t\in\rpn}\mathcal G_t$.
Moreover, using that $\mathbb E(M_\infty(p))>0$ one obtains that $M_\infty(p)>0$ $\mathbb P$--a.s., see Lemma~1.35 of \cite{Kno11} (or Theorem 2 of  \cite{Be4} for the conservative case). Consequently, restricted to the $\sigma$--algebra $\mathcal G_\infty$, the measures $\mathbb P^{(p)}$ and $\mathbb P$ are equivalent. 
\end{rem}
Corollary~3.10 in \cite{Kyp06} shows that under the measure $\mathbb P^{(p)}$ the process $X_1$ is again a spectrally negative L\'evy process such that 
\begin{equation}
\psi_p(\lambda):= \frac{1}{t}\log\mathbb{E}^{(p)}(e^{\lambda X_1(t)}) = \psi(\lambda+p) - \psi(p) = c\lambda - \Phi(\lambda+p)+\Phi(p)
\label{diff-this}
\end{equation}
for all $\lambda >\underline{p}-p$. Let $W_p$\label{p.W_p} be the scale function of the spectrally negative L\'evy process $X_1$ under  $\mathbb P^{(p)}$. That is to say,  $W_p$ is the unique increasing and continuous function on $(0,\infty)$ that is defined through the Laplace transform
\[
\int_0^\infty e^{-\lambda x}W_p(x){\rm d}x=\frac{1}{\psi_p(\lambda)},
\]
for all  $\lambda>p-\underline{p}$. 

A fundamental identity involving the scale function $W_p$ that we shall appeal to later is the following result taken from Theorem 8.1, equation (8.7), in  \cite{Kyp06}:
\begin{equation}
\mathbb{P}^{(p)}(\tau^-_{1,x}=\infty) = (\psi'_p(0+)\vee 0)W_p(x)
\label{8.15}
\end{equation}
for all $x>0$.
Another important fact that we shall also use  concerns the value of $W_p$ at zero. Indeed, thanks to the fact that $X_1$ has paths of bounded variation, it turns out that for all $p\geq 0$, $W_p(0+) = 1/c$.
See for example Lemma~8.6 in \cite{Kyp06}.

An important role in what follows will be played by $X_n$ killed upon hitting $(-\infty,-x)$ for $n\in\N$ and $x\in\rpn$. 
For $t\in\rpn
$ set 
\[
X^x_n(t):=(X_n(t)+x)\mathds1_{\{\tau^-_{n,x}>t\}}=\left(x+ct+\log|B_n(t)|\right)\mathds1_{\{\tau^-_{n,x}>t\}}. 
\]

\section{Properties of the extinction probability}\label{s.properties}

In this section we prove Theorem~\ref{p.ext-probab} by dealing  with the cases $c\in(0, c_{\bar p}]$ and $c>c_{\bar p}$ as two separate lemmas. The first lemma below deals with the easier, but less interesting, case that $c\in(0,c_{\bar p}]$.

\begin{lemma}\label{t.sc.1}
Let $c\in(0,c_{\bar p}]$. Then $\mathbb P(\zeta^x<\infty)=1$ for all $x\in\rpn$.
\end{lemma}

\begin{proof}  Using stochastic monotonicity it suffices to consider the case that $c = c_{\bar p} =\Phi({\bar p})/(1+{\bar p})$.
 It was shown in Theorem~4 in \cite{BR05} (cf. also Theorem~1 in \cite{BR03}) that $M_t(\bar p)\to0$ $\mathbb P$--a.s. as $t\to\infty$.  Since $M_t(\bar p)\ge e^{\Phi(\bar p)t}\lambda^{1+{\bar p}}_1(t)
$ for all $t\in\rpn$, we thus deduce that 
\[
\left(c_{\bar p}t+\log(\lambda_1(t))\right)\to-\infty
\]
as $t\to\infty$ and hence $\mathbb P(\zeta^x<\infty)=1$ for all $x\in\rpn$. 
\rule{1cm}{0cm}
\end{proof}

Notice that the statement of the previous lemma is obvious for $c\in(0,c_{\bar p})$ as the asymptotic decay of the largest fragment in the non--killed setting is given by $c_{\bar p}$, see Proposition~\ref{largestfragmentnotkilled}, and thus the fragmentation process eventually crosses the killing line almost surely. However, for the critical value $c=c_{\bar p}$ this argument does not work as one needs to rule out the possibility that the largest fragment could approach the killing line without intersecting it.

The following result deals with the more interesting case that $c>c_{\bar p}$. 

\begin{lemma}\label{positivesurviaval}
Let $c>c_{\bar p}$. Then 
\[
\mathbb P(\zeta^x<\infty)\in(0,1)
\] 
for all $x\in\rpn$.
\end{lemma}

\begin{proof}
The proof is divided into two parts. The first part shows that $\mathbb P(\zeta^x<\infty)<1$ and the second part proves that $\mathbb P(\zeta^x<\infty)>0$ for all $x\in\rpn$.

\underline{Part I}.
Note that $c>c_{\bar{p}} = \Phi'(\bar{p})$ and hence, since $\Phi'$ is continuous, we may always choose $p\in(\underline p,\bar p)$ such that $c>\Phi'(p)$. In that case  $\psi_p'(0+)  = \psi'(p) = c- \Phi'(p)>0$. 
Hence, by means of the nondecreasingness of $\mathbb P(\zeta^{(\cdot)}=\infty)$, we deduce from (\ref{8.15}) 
that 
\[
\mathbb P^{(p)}(\zeta^x=\infty)\ge\mathbb P^{(p)}(\tau^-_{1,0}=\infty)=\psi_p'(0+)W_p(0+) = \frac{\psi_p'(0+)}{c}\in(0,1)
\]
for all $x\in\rpn$. According to Remark~\ref{r.equivalentmeasures} this results in 
\[
\label{e.positivesurviaval.1}
\mathbb P(\zeta^x=\infty)>0,\qquad\text{i.e.}\qquad\mathbb P(\zeta^x<\infty)<1.
\]

\underline{Part II}.
Let $x\in\rpn$. In order to show that $\mathbb P(\zeta^x<\infty)>0$ we fix some $a>x$ and some $y_0\in(\nicefrac{1}{2}\lor(1-e^{-a}),1)$ such that 
\[
q:=\mu(\pi\in\mathcal P:|\pi_1|\in(0, y_0])\in(0,\infty).
\] 
The last inclusion is possible since, on the one hand, $\mu(\pi\in\mathcal P:|\pi_1|\in(0,y_0]) = \mu(\pi\in\mathcal{P}: -\log |\pi_1|\geq -\log y_0)$ is the tail of the jump measure of the subordinator $\xi$ which is necessarily finite.
On the other hand, there exists some $x\in(0,1)$ such that 
\[
\mu(\pi\in\mathcal P:|\pi_1|\in(0,x])>0,
\]
as otherwise the L\'evy measure of $\xi$ has no mass in $(0,\infty)$ which contradicts the fact that $\xi$ is a subordinator.


Recall that $\{\pi(t): t\in\mathcal I_1\}$ are the atoms of the Poisson point process on $\mathcal{P}$ that determines $\xi$. Further, denote the (possibly infinite) killing time of $\xi$ by $\tau_\xi$ and recall from (\ref{e.kappa}) that $\tau_\xi$ is independent of the dynamics of the process $\xi$, up to its moment of killing, and exponentially distributed with parameter $\kappa$. Moreover, by means of Proposition~2 in Section 0.5 of  \cite{Ber96} we  have that $\tau(y_0):=\inf\{t\in \mathcal I_1:|\pi_1(t)|\in(0,y_0]\}$ is exponentially distributed with parameter $q$. 
It is  straightforward to check that every block which does not contain 1 and which is produced at some dislocation before the time $\tau(y_0)\wedge\tau_\xi$ of the block containing 1 will be no larger  than a proportion $e^{-a}$ of its parent.  This follows directly from the inequality that for all $t\in\mathcal{I}_1$ with $t<\tau(u_0)$ and all $j\in\mathbb{N}\backslash\{1\}$,
\[
|\pi_j(t)| \leq \sum_{n\in\mathbb{N}\backslash\{1\}} |\pi_n(t)|\leq 1- |\pi_1(t)|\leq 1-y_0\leq e^{-a}.
\]



The classical Thinning Theorem for Poisson point processes (e.g. Proposition~2 in Section 0.5 of \cite{Ber96}) allows us to conclude that $(X^x_1(u))_{u\in[0,\tau(y_0)\wedge\tau_\xi)}$   has the law of a L\'evy process, say $\widetilde X^x_1$, which is the difference of a linear drift with constant rate $c$ and a driftless subordinator with L\'evy measure $\mu(\pi\in\mathcal{P}: -\log |\pi_1|\in {\rm d}x)|_{(0, -\log y_0]}$, 
sampled up to  a time which is the minimum of two  independent and exponentially distributed random times, say $\mathbf{e}_q$ and $\mathbf{e}_\kappa$, with respective rates $q$ and $\kappa$.

Now define
\[
{R}^{(q+\kappa)}(a, x, {\rm d}y) = \int_0^\infty e^{-(q+\kappa)t} {\rm d}t\cdot \mathbb{P}\left(\widetilde{X}^x_1(t) \in{\rm d}y, \, \sup_{s\leq t}\widetilde{X}^x_1(s) \leq a,\, \inf_{s\leq t}\widetilde{X}^x_1(s) \geq 0\right),\qquad y\in(0,a).
\]
Theorem 8.7 in \cite{Kyp06} shows that ${R}^{(q+\kappa)}(a, x, {\rm d}y)$ is absolutely continuous with strictly positive Lebesgue density in the neighbourhood of the origin (this is at least immediately obvious for $y\in(0,x)$ by inspecting the expression for the resolvent in the aforementioned theorem). A little thought in light of the above remarks  reveals that,  on the event $\{\sup_{s< \tau(y_0)\wedge\tau_\xi}{X}^x_1(s) \leq a,\, \inf_{s< \tau(y_0)\wedge\tau_\xi}{X}^x_1(s) \ge 0\}$, the process $(X^x_1(u))_{u\in[0,\tau(y_0)\wedge\tau_\xi)}$ describes (on the negative-logarithmic scale and relative to the killing barrier) the {\bf only} surviving block in the process $\Pi^x$ over the time horizon $[0,\tau(y_0)\wedge\tau_\xi)$.

Using these facts, as well as the observation that $\tau(y_0)$ is almost surely not a jump time for $\widetilde{X}_1^x$, we now have the estimate
\begin{eqnarray*}
\mathbb{P}(\zeta^x<\infty) 
&\geq&  \mathbb P\left({X}^x_1(\tau(y_0)-)\in\left[0,-\log\left(y_0\right) \right),\,
\sup_{s< \tau(y_0)}{X}^x_1(s) \leq a,\, \inf_{s< \tau(y_0)}{X}^x_1(s) \geq 0,\, \tau(y_0)<\tau_\xi\right) \\
&\geq&  \mathbb P\left(\widetilde{X}^x_1(\mathbf{e}_q-)\in\left[0,-\log\left(y_0\right)\right),\,
\sup_{s< \mathbf{e}_q}\widetilde{X}^x_1(s) \leq a,\, \inf_{s< \mathbf{e}_q}\widetilde{X}^x_1(s) \geq 0,\, \mathbf{e}_q <\mathbf{e}_\kappa\right) \\
&\geq&  \mathbb E\left(e^{-\kappa\mathbf{e}_q} ; \widetilde{X}^x_1(\mathbf{e}_q-)\in\left[0,-\log\left(y_0\right)\right),\,
\sup_{s< \mathbf{e}_q}\widetilde{X}^x_1(s) \leq a,\, \inf_{s< \mathbf{e}_q}\widetilde{X}^x_1(s) \geq 0\right) \\
&=& q R^{(q+\kappa)}\left(a,x,\left[0,-\log\left(y_0\right)\right)\right)>0
\end{eqnarray*}
as required. 
\end{proof}

\section{Explosion of the number of blocks on survival}\label{proof of l.p.2.2}
In this section we provide the proof of Theorem~\ref{l.p.2.2}. To this end, we shall use the following auxiliary lemma which states that for any $n\in\N$ there exists a  time such that with positive probability the fragmentation process has at  least $n$ blocks. More precisely, we have the following result.

\begin{lemma}\label{l.l.p.2.2.1}
Let $c>c_{\bar p}$. Then for any $n\in\N$ there  exists a  $t>0$ such that
\begin{equation}\label{e.l.l.p.2.2.1.0}
\mathbb P\left(N^0_t\ge n\right)>0.
\end{equation}
\end{lemma} 

\begin{proof} 
In the first part of the proof we show that the probability of the event $\{N^0_t\ge2\}$ is positive for some $t\in\rpn$ and in the second part we use this in conjunction with an induction argument to prove the assertion.  

\underline{Part I}.
Let us first show that there exists some $z_0\in(\nicefrac{1}{2},1)$ such that 
\begin{equation}\label{e.z_0}
\mu(\pi\in\mathcal P:|\pi|^\downarrow_2 \ge 1-z_0, |\pi_1|>0)>0,
\end{equation}
where $\{|\pi|^\downarrow_i: i\geq 1\}$ represents the asymptotic frequencies of $\pi\in\mathcal{P}$ when ranked in descending order. 
To this end, assume $\mu(\pi\in\mathcal P:|\pi|^\downarrow_2 \ge a, |\pi_1|>0)=0$ for all $a\in(0,1)$. This assumption implies that $\mu(\pi\in\mathcal P:|\pi|^\downarrow_2  \neq 0, |\pi_1|>0)=0$, which in view of (\ref{e.levymeasure.0}) results in $\mu(\pi\in\mathcal P:|\pi_1|>0)=0$ and thus contradicts $\nu\ne0$. Consequently, there exists some $z_0\in(\nicefrac{1}{2},1)$ such that (\ref{e.z_0}) holds.
Next note that, on account of the inequality $|\pi|_1^\downarrow + |\pi|_2^\downarrow \leq 1$,
\[
p:=\mu(\pi\in\mathcal P:|\pi|^\downarrow_1 \leq z_0, |\pi_1|>0)\ge\mu(\pi\in\mathcal P:|\pi|^\downarrow_2 \ge 1-z_0, |\pi_1|>0)>0.
\]
Observe that 
$\nu({\bf s}\in\mathcal S:s_1\in(0,z_0])<\infty$,
as otherwise 
\[
\int_{\mathcal S}(1-s_1)\nu(\dd{\bf s})\ge\int_{\{{\bf s}\in\mathcal S:s_1\in(0,z_0]\}}(1-s_1)\nu(\dd{\bf s})\ge(1-z_0)\nu({\bf s}\in\mathcal S:s_1\in(0,z_0])=\infty,
\]  
which contradicts (\ref{e.levymeasure}). Therefore, we infer from formula (3) in \cite{HKK} that
\[
p\leq \mu(\pi\in\mathcal P:|\pi|^\downarrow_1 \leq z_0)\le 
\nu({\bf s}\in\mathcal{S}: s_1\in(0,z_0])<\infty.
\]

Now let $\eta(z_0) := \inf\{t\in\mathcal{I}_1: |\pi(t)|_1^\downarrow\leq z_0, |\pi_1(t)|>0\}$.
The classical Thinning Theorem for Poisson point processes 
shows that $|\pi(\eta(z_0))|^\downarrow$ and $\eta(z_0)$ are independent and that
\begin{equation}\label{e.l.l.p.2.2.1.0a}
\mathbb P\left(|\pi(\eta(z_0))|^\downarrow_2\ge1-z_0\right)=\frac{\mu(\pi\in\mathcal P:|\pi|^\downarrow_2\ge 1-z_0, |\pi_1|>0)}{\mu(\pi\in\mathcal P:|\pi|^\downarrow_1\le z_0, |\pi_1|>0)}>0.
\end{equation}
Moreover $|\pi(\eta(z_0))|^\downarrow$ and $(X^x_1(u))_{u\in[0,\eta(z_0))}$ are independent and $\eta(z_0)$ is 
   exponentially distributed with parameter $p$.
Now let $\hat X_1^x$ be a spectrally negative L\'evy process, shifted by $x\in\rpn$, which is written as the difference of 
a linear drift with rate $c$ and a driftless subordinator with L\'evy measure $\mu(\pi\in\mathcal{P}:  |\pi|_1^\downarrow > z_0;\,   -\log|\pi_1|\in{\rm d}y )$, $y>0$, which is independent of all other previously mentioned random objects. 

We want to work with its resolvent on the half line
\[
R^{(p+\kappa)}(x, {\rm d}y) :=\int_0^\infty e^{-(p+\kappa) t}{\rm d}t\cdot\mathbb{P}(
\hat X_1^x(t)\in{\rm d}y,\, \inf_{s\leq t}\hat X^x_1(s)\geq 0), \qquad y>0
\]
which is  known to have a strictly positive density for all $x\geq 0$, cf. Corollary 8.8 of \cite{Kyp06}. Note that in the case $x=0$ the process $\hat X_1^x$ will take a strictly positive amount of time to exit the domain $[0,\infty)$ on account of path irregularity, see the introduction of Chapter 8 of \cite{Kyp06}.
Let $\mathbf{e}_p$ and $\mathbf{e}_\kappa$ be two independent (of everything) exponentially distributed random variables with 
respective rates $p$ and $\kappa$.
Since almost surely neither $\mathbf{e}_p$ nor $\mathbf{e}_\kappa$ is a jump time for $\hat{X}^0_1$, it follows from (\ref{e.l.l.p.2.2.1.0a}) that
\begin{align*}
& \mathbb P\left(N^0_{\eta(z_0)}\ge2\right)\notag 
\\[0.5ex]
&\geq \mathbb P\left({X}^0_1(\eta(z_0)-)> -\log (1-z_0),\, |\pi(\eta(z_0))|^\downarrow_2\ge 1- z_0,\, \eta(z_0)<\tau_\xi\right)\notag
\\[0.5ex]
&= \mathbb E\left(e^{-\kappa \mathbf{e}_p}; \hat X^0_1(\mathbf{e}_p-)> -\log(1-z_0), \inf_{s<\mathbf{e}_p} \hat X_1^0(s)\geq 0\right)\mathbb P\left(|\pi(\eta(z_0))|^\downarrow_2\ge1-z_0\right)
\notag
\\[0.5ex]
%
&= pR^{(p+\kappa)}\left(0,(-\log (1-z_0), \infty)\right)\frac{\mu(\pi\in\mathcal P:|\pi|^\downarrow_2\ge 1-z_0, |\pi_1|>0)}{\mu(\pi\in\mathcal P:|\pi|^\downarrow_1\le z_0, |\pi_1|>0)}
>0.
\end{align*}
Given that $\eta(z_0)$ is exponentially distributed, it is now a standard argument to deduce that there must exist a $t>0$ such that
\begin{equation}\label{e.l.p.2.2.1a}
\mathbb P\left(N^0_t\ge2\right)>0.
\end{equation}

\underline{Part II}.
We prove (\ref{e.l.l.p.2.2.1.0}) by resorting to the principle of mathematical induction. 
To this end, let  $n\in\N$, fix some $u_0>0$ such that (\ref{e.l.p.2.2.1a})  holds and, as the induction hypothesis, assume that
\[
\mathbb P(N^0_{nu_0}\ge n+1)>0.
\]
To provide an estimate for $\mathbb P(N^0_{(n+1)u_0}\ge n+2)$ note that the event $\{N^0_{(n+1)u_0}\ge n+2\}$ contains the event that $N^0_{nu_0}\ge n+1$ and subsequently $n$ of the blocks alive at time $n u_0$ survive for a further $u_0$ units of time, whilst one of the blocks at time $nu_0$ succeeds in fragmenting further to produce at least two further particles $u_0$ units of time later. A lower bound on the probability of the latter event that makes use of the fragmentation property and the monotonicity  in $x$ of $\mathbb{P}(N^x_{nu_0}\ge n+1)$ and $P\left(\zeta^x>u_0\right)$, produces the estimate,
\begin{align*}
\mathbb P\left(N^0_{(n+1)u_0}\ge n+2\right) 
&\ge 
\mathbb P\left(N^0_{nu_0}\ge n+1\right)\mathbb P\left(N^0_{u_0}\ge2\right)\mathbb P\left(\zeta^0>u_0\right)^n
> 0.
\end{align*}
Coupled with (\ref{e.l.p.2.2.1a}), which closes the argument by induction, the proof of the lemma is complete. 
\end{proof}

Having established the previous lemma we are now in a position to tackle the proof of Theorem~\ref{l.p.2.2}.

{\it Proof of Theorem~\ref{l.p.2.2}.} By Lemma \ref{l.l.p.2.2.1},
fix some $k\in\N$ as well as $t_0>0$ such that $\mathbb P\left(N^0_{t_0}\ge k\right)>0$ and for every $n\in\N$ and $x\in\rpn$ define 
\[
E^x_n:=\left\{\omega\in\Omega:N^x_{nt_0}(\omega)\ge k\right\}.
\] 
By means of the fragmentation property and  
the monotonicity  in $x$ of $\mathbb P\left(N^x_{t_0}\ge k\right)$
\begin{equation}\label{e.l.p.2.2.4}
\mathbb P\left(E^x_n\left|\mathcal F_{(n-1)t_0}\right.\right)
\ge \mathbb P(N^0_{t_0}\ge k)>0
\end{equation}
on $\{\zeta^x=\infty\}$.
As a consequence of (\ref{e.l.p.2.2.4}) we obtain that
\begin{equation}\label{e.BC}
\sum_{n\in\N}\mathbb P\left(\left.E^x_n\right|\mathcal F_{(n-1)t_0}\right)=\infty
\end{equation}
$\mathbb P$--a.s. on $\{\zeta^x=\infty\}$ for any $x\in\rpn$. 

Since $E^x_n$ is $\mathcal F_{nt_0}$--measurable, we can apply the extended Borel--Cantelli lemma (see e.g.  Corollary (3.2) in Chapter~4 of \cite{Dur91} or Corollary~5.29 in \cite{Bre92}) to deduce that 
\[
\{E^x_n\text{ happens infinitely often}\}=\left\{\sum_{n\in\N}\mathbb P\left(\left.E^x_n\right|\mathcal F_{(n-1)t_0}\right)=\infty\right\},
\]
and thus (\ref{e.BC}) shows that on the event $\{\zeta^x=\infty\}$, $x\in\rpn$, the event $E^x_n$ happens for infinitely many $n\in\N$. 
Consequently, we infer by monotonicity in $x$ of $N^x_t$ that
\[
\mathbb P\left(\left.\limsup_{t\to\infty}N^x_t\ge k\right|\zeta^x=\infty\right)=1,
\]
which proves the assertion on account of the fact that $k$ may be taken arbitrarily large.
\hfill$\square$

\section{Multiplicative martingales}\label{s.mm}

Like many different types of spatial branching processes, the probability of extinction of our killed fragmentation process turns out to be  intimately related to certain product martingales which we now introduce. 

More specifically, the object under consideration in the present section is the stochastic process defined as follows. For any function $f:\R\to[0,1]$ and $x\in\rpn$ let $Z^{x,f}:=\{Z^{x,f}_t:t\geq 0\}$\label{p.Z^x^f} be given by
\[
Z^{x,f}_t=\prod_{n\in\mathcal N^x_t}f\left(x + ct + \log |\Pi^x_n(t)|\right),\, t\geq 0.
\] 
We are interested in understanding which functions $f$ make the above process a martingale. In that  case we refer to $Z^{x,f}$ as a {\it multiplicative martingale}. The following theorem shows that within the class of nonincreasing functions which are valued zero at $\infty$, there is a unique choice of $f$. 

\begin{theorem}\label{t.p.t.1.2.1.T}
Let $c>c_{\bar p}$ and let $f:\R\to[0,1]$ be a monotone function. Then the following two statements are equivalent.
\begin{itemize}
\item[(i)] For any $x\in\rpn$ the process $Z^{x,f}$ is a martingale with respect to the filtration $\mathcal{F}$ and 
\[
\lim_{x\to\infty}f(x)=0.
\]
\item[(ii)] For all $x\in\rpn:$
\[
f(x)=\mathbb P\left(\zeta^x<\infty\right).
\]
\end{itemize}
\end{theorem} 

For any $c>c_{\bar p}$ and $t,x\in\rpn$ define 
\[
R^x_1(t) = x + ct +\log \lambda_1(t).
\] 
In order to prove Theorem~\ref{t.p.t.1.2.1.T} we shall use the following lemma which states that on survival of the killed fragmentation process the process $(R^x_1(t))_{t\in\rpn}$ is unbounded.

\begin{lemma}\label{p.2}
Let $c>c_{\bar p}$ and $x\in\rpn$. 
Then we have
\[
\limsup_{t\to\infty}R^x_1(t)=\infty
\]
$\mathbb P(\cdot|\zeta^x=\infty)$--almost surely.
\end{lemma}

\begin{proof}
Let $z>x$ and set 
\[
\Gamma^x_z:=\text{\Large\{}\omega\in\Omega:\inf\{t\in\rpn:X^x_n(t)(\omega)\not\in [0,z)\}=\infty\,\forall\,n\in\N\text{\Large\}}.
\] 
Theorem~12 in Section~VI.3 of \cite{Ber96} shows that the probability that a spectrally negative L\'evy process never leaves the  interval $(0,z)$  when started in its interior is zero. Consequently, we have that 
\[
\tau^-_{n,x}<\tau^+_{n,z-x}=\infty\quad\text{on}\quad \Gamma^x_z.
\] 
For each $n\in\N$ set  
\[
\sigma_n:=\inf\{t\in\rpn:N^x_t\ge n\}
\]
and note that Theorem~\ref{l.p.2.2} implies that $\sigma_n$ is a $\mathbb P$--a.s. finite stopping time on $\{\zeta^x=\infty\}$. Let $\widetilde{\mathcal{N}}_t^x = \{n\in\mathbb{N}: X_n^x(t) \geq 0\}$ and introduce the equivalence relation $\sim$ on $\widetilde{\mathcal{N}}_t^x$  such that $n\sim m$ when $n\in B_m(t)$.
The cardinality of $\widetilde{\mathcal{N}}^x_t/\sim$ is equal to $N^x_t$. Further, let $p\in(\underline p,\bar p)$. 
By means of Lemma~8.6 of \cite{Kyp06}, we then infer from the strong fragmentation property and equation (8.8) of Theorem~8.1 of \cite{Kyp06} that
\begin{align*}
\mathbb P^{(p)}(\Gamma^x_z|\mathcal F_{\sigma_n})\notag 
&\le \prod_{k\in\widetilde{\mathcal N}^x_{\sigma_n}/\sim}\left.\mathbb P^{(p)}(\Gamma^y_z)\right|_{y=X^x_k({\sigma_n})} \notag
\\[0.5ex]
&= \prod_{k\in\widetilde{\mathcal N}^x_{\sigma_n}/\sim}\left.\mathbb P^{(p)}(\tau^-_{k,y}<\tau^+_{k,z-y})\right|_{y=X^x_k({\sigma_n})}
\\[0.5ex]
&= \prod_{k\in\widetilde{\mathcal N}^x_{\sigma_n}/\sim}\left(1-\frac{W_p(X^x_k({\sigma_k}))}{W_p(z)}\right)\notag
\\[0.5ex]
&\le \left(1-\frac{1}{cW_p(z)}\right)^{N^x_{\sigma_n}}\notag
\\[0.5ex]
&\le \left(1-\frac{1}{cW_p(z)}\right)^n\notag
\end{align*}
$\mathbb P^{(p)}$--a.s. on $\{\zeta^x=\infty\}$ for any $n\in\N$. Therefore, since $\{R^x_1(s)<z\,\forall\,s\in\rpn\}=\Gamma^x_z$ ,we have
\begin{align*}
\mathbb P^{(p)}\left(\left\{\sup_{s\in\rpn}R^x_1(s)<z\right\}\cap\left\{\zeta^x=\infty\right\}\right) &= \mathbb P^{(p)}\left(\Gamma^x_z\cap\left\{\zeta^x=\infty\right\}\right)
\\[0.5ex]
&= \lim_{n\to\infty}\mathbb E^{(p)}\left(\left.\mathbb P^{(p)}\left(\Gamma^x_z\cap\left\{\zeta^x=\infty\right\}\right|\mathcal F_{\sigma_n}\right)\right)
\\[0.5ex]
&= \mathbb E^{(p)}\left(\lim_{n\to\infty}\left.\mathbb P^{(p)}\left(\Gamma^x_z\cap\left\{\zeta^x=\infty\right\}\right|\mathcal F_{\sigma_n}\right)\right)
\\[0.5ex]
&= 0.
\end{align*}
From this last equality and the fact that $z>x$ is arbitrary, one readily deduces that 
\[
\mathbb P^{(p)}\left(\left\{\limsup_{s\to\infty}R^x_1(s)<\infty\right\}\cap\left\{\zeta^x=\infty\right\}\right) 
= 0.
\]
Since both events $\{\limsup_{s\to\infty}R^x_1(s)<\infty\}$ and $\{\zeta^x=\infty\}$ are $\mathcal G_\infty$-measurable, we therefore infer from Remark~\ref{r.equivalentmeasures} that
\[
\mathbb P\left(\left\{\limsup_{s\to\infty}R^x_1(s)<\infty\right\}\cap\left\{\zeta^x=\infty\right\}\right)=0,
\]
which proves the assertion.
\end{proof}

Let us now tackle the proof of Theorem~\ref{t.p.t.1.2.1.T}.

{\it Proof of Theorem~\ref{t.p.t.1.2.1.T}.}
The proof is guided by a similar result in Harris et al. \cite{harrisetal} for branching Brownian motion. We divide the proof into two parts.  The first part proves the uniqueness of  monotone functions $f$  satisfying $\lim_{y\to\infty}f(y) = 0$  for which $Z^{x,f}$ is a martingale. Part II of the proof shows that the probability of extinction constitutes a function that makes $Z^{x,f}$ a martingale.

\underline{Part I}.
By the martingale convergence theorem we have that $Z^{x,f}$ being a 
nonnegative martingale implies that $Z^{x,f}_\infty:=\lim_{t\to\infty}Z^{x,f}_t$\label{p.Z^x^f_inf} exists $\mathbb P$--almost surely. Since the empty product equals 1 it is immediately clear that 
\begin{equation}\label{e.l.p.t.1.2.1.1}
Z^{x,f}_\infty=1
\end{equation} 
holds $\mathbb P$--a.s. on $\{\zeta^x<\infty\}$. Moreover, according to Lemma~\ref{p.2} we have that $\limsup_{t\to\infty}R^x_1(t)=\infty$ $\mathbb P$--a.s. on $\{\zeta^x=\infty\}$. Since $\lim_{y\to\infty}f(y)=0$, we thus deduce that
\begin{equation}\label{e.l.p.t.1.2.1.2}
0\le Z^{x,f}_\infty\le\liminf_{t\to\infty}f(R^x_1(t))
=0
\end{equation} 
$\mathbb P$--a.s. on $\{\zeta^x=\infty\}$. Hence, in view of (\ref{e.l.p.t.1.2.1.1}) and (\ref{e.l.p.t.1.2.1.2}) we infer that
\begin{equation}\label{e.l.p.t.1.2.1.3}
Z^{x,f}_\infty=\mathds 1_{\{\zeta^x<\infty\}}
\end{equation}
holds true $\mathbb P$--almost surely. As a consequence of $Z^{x,f}$ being a bounded, and hence uniformly integrable, martingale we conclude from (\ref{e.l.p.t.1.2.1.3}) that
\[
f(x)=\mathbb E(Z^{x,f}_0)=\mathbb E(Z^{x,f}_\infty)=\mathbb P(\zeta^x<\infty).
\]

\underline{Part II}.
Recalling Lemma \ref{positivesurviaval}, let $g:\R\to(0,1)$ be given by $g(x)=\mathbb P(\zeta^x<\infty)$. Since $g$ is monotone and bounded, the limit $g(+\infty): =\lim_{x\to\infty}g(x)$ exists in $[0,1)$. Furthermore, for any $t\in\rpn$ we have 
$\lim_{x\to\infty}\mathds1_{\mathcal N^x_t}(n)=1$ $\mathbb P$--a.s. for every $n\in\N$. In addition, we have that $X^x_n(t)\to\infty$ $\mathbb P$--a.s. for any $n\in\N$  and $t\in\rpn$ as $x\to\infty$. Resorting to the fragmentation property of $\Pi$ we deduce that  
\begin{equation}\label{e.l.e.FKPP.2.1.1.0a.b}
g(x)=\mathbb E\left(\mathbb P\left(\left.\zeta^x<\infty\right|\mathcal F_t\right)\right)=\mathbb E\left(\prod_{n\in\mathcal N^x_t}g\left(x + ct +\log |\Pi^x_n(t)|\right)\right)=\mathbb E\left(Z^{x,g}_t\right)
\end{equation}  
holds for all $t\in\rpn$. 
By means of the fragmentation property we thus have that
\[
\mathbb E\left(\left.Z^{x,g}_{t+s}\right|\mathcal F_t\right)
=
\prod_{n\in\mathcal N^x_t}g(x + ct +\log |\Pi^x_n(t)|)=Z^{x,g}_t
\]
$\mathbb P$--almost surely. 
Hence, $Z^{x,g}$ is a $\mathbb P$--martingale. Moreover, by the  Dominated Converge Theorem,  we deduce from (\ref{e.l.e.FKPP.2.1.1.0a.b}) that
\begin{align*}
g(+\infty) &= \lim_{x\to\infty}\mathbb E\left(\prod_{n\in\mathcal N^x_t}g(x + ct +\log |\Pi^x_n(t)|)\right)
\\[0.5ex]
&= \mathbb E\left(\lim_{y\to\infty}\prod_{n\in\mathcal N^y_t}\lim_{x\to\infty}g(x)\right)
\\[0.5ex]
&= \mathbb E\left(\lim_{y\to\infty}\lim_{x\to\infty}g(x)^{N^y_t}\right).
\end{align*}
Consequently,  
$
g(+\infty)\in\{0,1\}.
$
Since $g$ is decreasing and $g(x)\in(0,1)$ for all $x\in\rpn$, this forces us to choose  $g(+\infty)=0$.
\hfill$\square$

\section{Additive martingales}\label{s.am}
In this section we deal with an additive stochastic process $M^x(p):=(M^x_t(p))_{t\in\rpn}$, $p\in(\underline p,\infty)$, that for $c>c_{\bar p}$ and $x\in\rpn$, is given by
\begin{align*}
M^x_t(p) &= \sum_{n\in\mathcal N^x_t}W_p\left(x + ct +\log |\Pi^x_n(t)|\right)e^{\Phi(p)t}|\Pi^x_n(t)|^{1+p}.
\end{align*}

The main result of this section is the following theorem. 

\begin{theorem}\label{t.mainresult.2}
Let $c>c_{\bar p}$ and let $p\in(\underline p,\bar p)$ be such that $c>\Phi'(p)$. Then the process $M^x(p)$ is a nonnegative $\mathcal F$--martingale with $\mathbb P$--a.s. limit $M^x_\infty(p)$. Moreover, this martingale limit satisfies
\[
\mathbb P\left(\{M^x_\infty(p)=0\}\triangle\{\zeta^x<\infty\}\right)=0,
\]
where $\triangle$\label{p.sd} denotes the symmetric difference.
\end{theorem}

The following lemma is a version of the so-called {\it many-to-one identity}. To state it, let us introduce, for each $n\in\mathbb{N}$ and $t\geq 0$, the notation $\{\overleftarrow{\Pi}_n(s): s\leq t\}$ to mean ancestral evolution of the block $\Pi_i(t)$. That is to say, if $\Pi_n(t)$ is the block containing $k\in\mathbb{N}$, then $\{\overleftarrow{\Pi}_n(s): s\leq t\} = \{B_k(s): s\leq t\}$.


\begin{lemma}\label{l.many-to-one}
We have
\[
\mathbb E\left(\sum_{n\in\N}|\Pi_n(t)|f(\{|\overleftarrow{\Pi}_n(s)|:s\le t\})\right)=\mathbb E\left(f(\{|\Pi_1(s)|:s\le t\})\right)
\]
for every $t\in\rpn$ and $f:\rcll([0,t],[0,1])\to\rpn$, where $\rcll$ denotes the space of c\`adl\`ag functions.
\end{lemma}


\begin{proof}
The proof follows directly as a consequence of the fact that $\Pi_1(t)$ has the law of a size-biased pick from $\Pi(t)$. See for example Lemma 2 of Berestycki et al. \cite{Betal}.
\end{proof}

The next lemma establishes the first assertion of Theorem~\ref{t.mainresult.2} in that it shows that under $\mathbb P$ the process $M^x(p)$ is a martingale for suitable $c$ and $p$.

\begin{lemma}\label{l.0}
Let $c>c_{\bar p}$ and let $p\in(\underline p,\bar p)$ be such that $c>\Phi'(p)$. Further, let $x\in\rpn$. Then the process $M^x(p)$ is a $\mathbb P$--martingale with respect to the filtration $\mathcal F$. 
\end{lemma}

\begin{proof}
Let us first show that for any $t\in\rpn$ the process $(W_p(X^x_1(s))\mathds1_{\{s<\tau^-_{1,x}\}})_{s\in\rpn}$ is a $\mathbb P^{(p)}$--martingale with respect to $\mathcal F$. It is a  straightforward exercise using (\ref{diff-this}) to show that  $\psi_p'(0+) = c - \Phi'(p)>0$. By the Markov property of $X_1$ under $\mathbb P^{(p)}$ we thus infer from (\ref{8.15}) that
\begin{eqnarray}\label{e.l.0.1}
\mathbb E^{(p)}\left(\left.\mathds1_{\{\tau^-_{1,x}=\infty\}}\right|\mathcal F_s\right) &=& 
\left.\mathbb P^{(p)}\left(\tau^-_{1,y}=\infty\right)\right|_{y=x+X_1(s)}\mathds1_{\{s<\tau^-_{1,x}\}}\notag
\\
&=& \psi_p'(0+)W_p(x+X_1(s))\mathds1_{\{s<\tau^-_{1,x}\}}
\end{eqnarray}
holds $\mathbb P^{(p)}$--a.s. for any $s\in\rpn$. Note that the left--hand side of (\ref{e.l.0.1}) defines a closed $\mathbb P^{(p)}$--martingale. Further, observe that $x+X_1(s)=X^x_1(s)$ on the event ${\{s<\tau^-_{1,x}\}}$. 

By means of Lemma~\ref{l.many-to-one} we deduce that 
\begin{align}\label{e.l.0.2}
\mathbb E\left(M^x_t(p)\right) &= e^{\Phi(p)s}\mathbb E\left(\sum_{n\in\mathcal N^x_t}W_p\left(x + ct +\log |\Pi^x_n(t)|\right)e^{\Phi(p)t}|\Pi^x_n(t)|^{1+p}\right)\notag 
\\[0.5ex]
&=  \mathbb E\left(W_p(X^x_1(t))\mathds1_{\{t<\tau^-_{1,x}\}}e^{\Phi(p)t-p\xi(t)}\right)\notag 
\\[0.5ex]
&= \mathbb E^{(p)}\left(W_p(X^x_1(t))\mathds1_{\{t<\tau^-_{1,x}\}}\right)
\notag\\[0.5ex]
&= W_p(x)
\end{align}
for all $t\in\rpn$, where the final equality is a consequence of the above--mentioned martingale property of $(W_p(X^x_1(s))\mathds1_{\{s<\tau^-_{1,x}\}})_{s\in\rpn}$. 
In view of (\ref{e.l.0.2}) we infer from the fragmentation property of $\Pi$ that
\begin{align*}
\mathbb E\left(\left.M^x_{t+s}(p)\right|\mathcal F_t\right)
&= \sum_{n\in\mathcal{N}_t^x}e^{\Phi(p)t}|\Pi^x_n(t)|^{1+p}\mathbb E\left(\left.M^{(n)}\right|\mathcal F_t\right)
\\[0.5ex]
&= \sum_{n\in\mathcal{N}_t^x
}e^{\Phi(p)t}|\Pi^x_n(t)|^{1+p}W_p(x + ct +\log |\Pi^x_n(t)|)
\\[0.5ex]
&= M^x_t(p)
\end{align*}
$\mathbb P$--a.s. for all $s,t\in\rpn$, where conditional on $\mathcal F_t$ the $M^{(n)}$ are independent and satisfy
\[
\mathbb P\left(\left.M^{(n)}\in\cdot\right|\mathcal F_t\right)=\left.\mathbb P\left(M^y_s(p)\in\cdot\right)\right|_{y=x + ct +\log |\Pi^x_n(t)|}
\]
$\mathbb P$--almost surely.
\end{proof}

Let us now turn to the proof of Theorem~\ref{t.mainresult.2}. The main ingredient in the proof of Theorem~\ref{t.mainresult.2} turns out to be Theorem~\ref{t.p.t.1.2.1.T}, which deals with the product martingale $Z^{x,f}$.

{\it Proof of Theorem~\ref{t.mainresult.2}.}
According to Lemma~\ref{l.0} we have that $M^x(p)$ is a nonnegative martingale and by the Martingale Convergence Theorem it follows  that $M^x_\infty(p):=\lim_{t\to\infty}M^x_t(p)$\label{p.M^x_infty.2} exists $\mathbb P$--almost surely.
It remains to show that the symmetric difference $\{M^x_\infty(p)=0\}\triangle\{\zeta^x<\infty\}$ is a $\mathbb P$--null set. 

Define the function $g_p:\rpn\to[0,1]$\label{p.gp} given by
\[
g_p(x)=\mathbb P(M^x_\infty(p)=0)
\]
for any $x\in\rpn$. Resorting to the fragmentation property we deduce that  
\[
\mathbb P(M^x_\infty(p)=0|\mathcal F_t)=\prod_{n\in\mathcal N^x_t}g_p\left(x + ct + \log |\Pi^x_n(t)|\right)=Z^{x,g_p}_t
\]
holds $\mathbb{P}$-almost surely for all $t\in\rpn$. 
Therefore, $Z^{x,g_p}$ is a $\mathbb P$--martingale. 
Note also that, thanks to the fact that both $\mathcal{N}^x_t$ and $W_p(x)$ are monotone increasing in $x$, for all $\epsilon>0$, $M^{x+\epsilon}_\infty(p)\geq M^x_\infty(p)$ and hence $g_p(\cdot)$ is a monotone function. It follows that $g_p(+\infty)$ exists in [0,1] and moreover, by taking the expectation and then the limit as $x\to\infty$ in (23), we infer that
\begin{equation}\label{e.1}
g_p(+\infty)\in\{0,1\}
\end{equation}
as otherwise we are led to the contradictory statement that $g_p(+\infty)<g_p(+\infty)$. 
 Recall the martingale $M(p)$ that we defined in Remark~\ref{r.equivalentmeasures}. Taking account of (\ref{8.15}) we have that $M^x_\infty(p)\leq M_\infty(p)/\psi'_p(0+)$. Hence, since $M(p)$ is an $L^q$-convergent martingale (cf. Theorem 2 of \cite{Be4} and Proposition~3.5 of \cite{Kno11}) for some $q>1$, it follows that $M^x(p)$ is too. 
Coupled with the stochastic monotonicity of $M^x_\infty(p)$ in $x$, this implies in view of (\ref{e.1}) that necessarily $g_p(+\infty) = 0$.

We may now apply Theorem~\ref{t.p.t.1.2.1.T} and infer that 
$
g_p(x)=\mathbb P(\zeta^x<\infty).
$
Since $\{\zeta^x<\infty\}\subseteq\{M^x_\infty(p)=0\}$ for each $x>0$ 
this implies that 
\[
\mathbb P\left(\{\zeta^x<\infty\}\triangle\{M^x_\infty(p)=0\}\right)=0
\]
for every $x>0$ as required.
\hfill$\square$

\section{Exponential decay rate of the largest fragment}\label{s.asympspeed}
The final section of this paper is devoted to the proof of Theorem~\ref{l.fp.1}. That is, in this section we deal with the asymptotic behaviour of the largest fragment in the killed fragmentation process.

{\it Proof of Theorem~\ref{l.fp.1}.} Our approach is based on the method of proof for Corollary~1.4 in \cite{bertfragbook} and makes use of the martingale $M^x(p)$ that we considered in the previous section.

First note that since $\lambda_1^x(t)\leq \lambda_1(t)$, it follows from Proposition \ref{largestfragmentnotkilled} that $\mathbb P$--a.s.,
\begin{equation}\label{e.l.fp.1.0b}
\limsup_{t\to\infty}\frac{1}{t}\log(\lambda^x_1(t))\le c_{\bar p}.
\end{equation}

Recall that we are assuming $c>c_{\bar p}$. In order to deal with the liminf  set 
\[
\hat p:=\inf\left\{p\in(\underline p,\bar p):\Phi'(p)< c\right\},
\] 
and let $p\in(\hat p,\bar p)$ as well as $\epsilon\in(0,p-\hat p)$. Since 
\[
\psi_p'(0+)= c-\Phi'(p)>\psi'_{p-\epsilon}(0+) =c-\Phi'(p-\epsilon)>c-\Phi'(\hat p)\ge0,
\] 
where $\Phi'(-\infty):=0$, we infer from (\ref{8.15}) that the scale functions $W_p$ and $W_{p-\epsilon}$ are uniformly bounded from above by $1/\psi'_p(0+)$ and $1/\psi'_{p-\epsilon}(0+)$, respectively. Moreover, according to Lemma~8.6 in \cite{Kyp06} we have $W_{p}(0+) = W_{p-\epsilon}(0+)=c^{-1}$. Hence,  there exists a constant $K>0$ such that $W_p(y)\leq K W_{p-\epsilon}(y)$ for all $y\geq 0$.
Observe that
\begin{align}\label{e.l.fp.1.1}
M_t^x(p)&
=\sum_{n\in\mathcal N^x_t}W_p\left(x + ct +\log |\Pi^x_n(t)|\right)e^{\Phi(p)t}|\Pi^x_n(t)|^{1+p}
\notag\\[0.5ex]
&\leq Ke^{(\Phi(p)-\Phi(p-\epsilon))t}[\lambda^x_1(t)]^\epsilon e^{\Phi(p-\epsilon)t}\sum_{n\in\mathcal N^x_t}W_{p-\epsilon}\left(x+ct+\log|\Pi^x_n(t)|\right)|\Pi^x_n(t)|^{1+p-\epsilon}\notag
\\[0.5ex]
&=K e^{(\Phi(p)-\Phi(p-\epsilon))t}[\lambda^x_1(t)]^\epsilon M^x_t(p-\epsilon).
\end{align}
According to Theorem~\ref{t.mainresult.2} we have that both  $M^x_\infty(p-\epsilon)$ and $M^x_\infty(p)$ are $(0,\infty)$-valued  $\mathbb P(\cdot|\zeta^x=\infty)$--almost surely. Consequently, taking the logarithm, dividing by $t$ and taking the limit inferior as $t\to\infty$ we thus deduce from (\ref{e.l.fp.1.1}) that 
\[
\liminf_{t\to\infty}\frac{1}{t}\log(\lambda^x_1(t))\ge-\frac{\Phi(p)-\Phi(p-\epsilon)}{\epsilon}
\]
$\mathbb P(\cdot|\zeta^x=\infty)$--almost surely. Therefore, we have
\begin{equation}\label{e.l.fp.1.2}
\liminf_{t\to\infty}\frac{1}{t}\log(\lambda^x_1(t))\ge-\lim_{\varepsilon\to0}\frac{\Phi(p)-\Phi(p-\varepsilon)}{\varepsilon}=-\Phi'(p)
\end{equation}
$\mathbb P(\cdot|\zeta^x=\infty)$--almost surely. Letting $p\to\bar p$ and resorting to the fact that $\Phi$ is the Laplace exponent of $\xi$, which ensures the continuity of $\Phi'$, (\ref{e.l.fp.1.2}) results in 
\begin{equation}
\liminf_{t\to\infty}\frac{1}{t}\log(\lambda^x_1(t))\ge-\Phi'(\bar p)
\label{bearinmind}
\end{equation}
$\mathbb P(\cdot|\zeta^x=\infty)$--almost surely. 

Recalling that $c_{\bar p}=\Phi'(\bar p)$, (\ref{e.l.fp.1.0b}) and (\ref{bearinmind}) imply the assertion of the theorem.
\hfill$\square$




\section*{Acknowledgements}

Both authors would like to thank Julien Berestycki and an anonymous referee for their valuable comments on this work.


\end{document}